\documentclass[12pt,twoside]{amsart}

\usepackage{amsmath,amsthm,amscd,amssymb,mathrsfs,graphicx,amsfonts,mathrsfs}
\usepackage{amsfonts}
\usepackage{amssymb,enumerate}
\usepackage{amsthm}
\usepackage[all]{xy}
\usepackage{hyperref}
\allowdisplaybreaks[4] \footskip=12pt
\renewcommand{\uppercasenonmath}[1]{}

\numberwithin{equation}{section} \theoremstyle{plain}
\newtheorem*{thm*}{Main Theorem}
\newtheorem{thm}{Theorem}[section]
\newtheorem{cor}[thm]{Corollary}
\newtheorem*{cor*}{Corollary}
\newtheorem{lem}[thm]{Lemma}
\newtheorem*{lem*}{Lemma}

\newtheorem*{fact*}{Fact}

\newtheorem*{nota*}{Notation}

\newtheorem*{prop*}{Proposition}

\newtheorem*{rem*}{Remark}

\newtheorem*{observation*}{Observation}

\newtheorem*{exa*}{Example}

\newtheorem*{df*}{Definition}

\newtheorem*{conj*}{Conjecture}

\renewcommand{\geq}{\geqslant}
\renewcommand{\leq}{\leqslant}

\begin{document}
\begin{center}
{\large  \bf Coartinianess of extension and torsion functors}

\vspace{0.3cm} Jingwen Shen, Xiaoyan Yang \\
Department of Mathematics, Northwest Normal University, Lanzhou 730070,
China\\
E-mails: shenjw0609@163.com, yangxy@nwnu.edu.cn
\end{center}
\bigskip
\centerline { \bf  Abstract} Let $(R,\mathfrak{m})$ be a commutative noetherian local ring with $\mathfrak{m}$-adic topology, $I$ an ideal of $R$. We investigate coartinianess of $\mathrm{Ext}$ and $\mathrm{Tor}$, show that the $R$-module $\mathrm{Ext}_{R}^{i}(N,M)$ is $I$-coartinian if $M$ is a linearly compact $I$-coartinian $R$-module and $N$ is an $I$-cofinite $R$-module of dimension at most $1$; the $R$-module $\mathrm{Tor}_{i}^{R}(N,M)$ is $I$-coartinian in the case $M$ is semi-discrete linearly compact $I$-coartinian and $N$ is finitely generated with dimension at most $2$.
\leftskip10truemm \rightskip10truemm \noindent \\
\vbox to 0.3cm{}

\leftskip0truemm \rightskip0truemm
\bigskip

\section{\bf Introduction and Preliminaries}
\renewcommand{\thethm}{\Alph{thm}}
 Let $R$ denote a commutative noetherian ring and $I$ an ideal of $R$. For an $R$-module $M$, the $i$th \emph{local cohomology module} of $M$ with respect to $I$ is defined as
\begin{center}$\begin{aligned}
\mathrm{H}^{i}_{I}(M)=\underrightarrow{\text{lim}}\mathrm{Ext}^{i}_{R}(R/I^{t},M).
\end{aligned}$\end{center}
An important conjecture of local cohomology is due to Grothendieck \cite{G}: The $R$-module $\mathrm{Hom}_{R}(R/I,\mathrm{H}_{I}^{i}(M))$ is finitely generated for every ideal $I$ and finitely generated $R$-module $M$. In 1970, Hartshorne \cite{H} provided a counterexample to show that this conjecture does not have an affirmative answer in general. He then defined an $R$-module $M$ to be \emph{$I$-cofinite} if $\mathrm{Supp}_{R}M\subseteq \mathrm{V}(I)$ and $\mathrm{Ext}^{i}_{R}(R/I,M)$ is finitely generated for all $i$, and asked:

\vspace{2mm}
{\it{Are the modules $\mathrm{H}^i_I(M)$ $I$-cofinite for all finitely generated
$R$-modules $M$ and all $i$?}}
\vspace{2mm}

Concerning this question, Hartshorne \cite{H} showed that if $R$ is a complete regular local ring and $I$ a prime ideal with 
$\mathrm{dim}R/I=1$ then $\mathrm{H}^i_I(M)$ is $I$-cofinite for any
finitely generated $R$-module $M$ and $i\geq0$.
Serval authors later extended and polished Hartshorne's work for the various case of $\mathrm{dim}R/I$ and the ring $R$ of lower dimensions (see \cite{B, BN, C, M2, Y3}).
On the other hand, Huneke and Koh \cite{HK} proved that for each pair of finitely generated modules $N$ and $M$ over a local ring $R$, under some special conditions, the $R$-module $\mathrm{Ext}_{R}^{1}(N,\mathrm{H}_{I}^{i}(M))$ is $I$-cofinite whenever $\mathrm{dim}R/I \leq 1$. It is well-known that if $\mathrm{dim}R/I \leq 1$ then $\mathrm{H}_{I}^{i}(M)$ are $I$-cofinite and $\mathrm{dim}_{R}\mathrm{H}_{I}^{i}(M)\leq 1$ for all finitely generated
modules $M$ and $i\geq0$. More general, Abazari and Bahmanpour \cite{AB,B1} proved that the $R$-module $\mathrm{Ext}_{R}^{i}(N,M)$ is $I$-cofinite for $i\geq 0$ if $N$ is finitely generated and $M$ is $I$-cofinite with $\mathrm{dim}_{R}M\leq 1$; the $R$-module $\mathrm{Tor}^{R}_{i}(N,M)$ is $I$-cofinite for $i\geq 0$ if $N$ and $M$ are both $I$-cofinite and $\mathrm{dim}_{R}N\leq 1$.

We denote $\Lambda_{I}(M)=\underleftarrow{\text{lim}}M/I^{t}M$ the $I$-adic completion of $M$, and recall the $i$th \emph{local homology module} of $M$ is
\begin{center}$\begin{aligned}
\mathrm{H}_{i}^{I}(M)=\underleftarrow{\text{lim}}\mathrm{Tor}_{i}^{R}(R/I^{t},M).
\end{aligned}$\end{center}Since torsion and
completion are dual, one expects these functors to live up to
behave in a manner dual to local cohomology.
Local homology as duality of local cohomoloy was initiated by Matlis \cite{M4} in 1974. As duality of cofinite modules,
Nam \cite{N} defined an $R$-module $M$ is \emph{$I$-coartinian} if $\mathrm{Cosupp}_{R}M\subseteq \mathrm{V}(I)$ and $\mathrm{Tor}_{i}^{R}(R/I,M)$ is artinian for all $i\geq 0$. The authors \cite{S2} considered two questions on coartinianness, which is in some sense dual to Hartshorne's questions on cofiniteness in \cite{H}, studied the coartinianess of local homology modules for semi-discrete linearly compact modules.

In the sequel the symbol $\mathcal{C}^{1}(R,I)_{cof}$ denotes the category of $I$-cofinite $R$-modules $N$ with $\mathrm{dim}_{R}N\leq 1$, $\mathcal{C}(R,I)_{coa}$ denotes the category of $I$-coartinian $R$-modules and $\mathcal{C}^{1}(R,I)_{coa}$ denotes the category of $R$-modules $M\in \mathcal{C}(R,I)_{coa}$ with $\mathrm{mag}_{R}M\leq 1$.
One attempt of the present paper is to study coartinianess of the $R$-modules $\mathrm{Ext}_{R}^{i}(N,M)$ and $\mathrm{Tor}^{R}_{i}(N,M)$. More precisely, we prove the following results:

\begin{thm}\label{thm:A}
Let $M$ be a linearly compact $R$-module. If $N\in \mathcal{C}^{1}(R,I)_{cof}$ and $M \in \mathcal{C}(R,I)_{coa}$, then $\mathrm{Ext}_{R}^{i}(N,M)\in \mathcal{C}^{1}(R,I)_{coa}$ for $i\geq 0$.
\end{thm}

\begin{thm}\label{thm:B}
Let $M$ be a semi-discrete linearly compact $R$-module and in $\mathcal{C}(R,I)_{coa}$ and $N$ a finitely generated $R$-module with $\mathrm{dim}_{R}N\leq 2$. Then the $R$-module $\mathrm{Tor}_{i}^{R}(N,M)$ is $I$-coartinian for $i\geq 0$.
\end{thm}

Next we list some notions which will need later.

Throughout this paper, $(R,\mathfrak{m})$ is a commutative noetherian local ring with $\mathfrak{m}$-adic topology and $I$ an ideal of $R$. Write $\mathrm{Spec}R$ for the set of prime ideals of $R$, and set
\begin{center}$\mathrm{V}(I)=\{\mathfrak{p}\in \mathrm{Spec}R\hspace{0.03cm}|\hspace{0.03cm}I\subseteq \mathfrak{p}\}.$\end{center}
Fix $\mathfrak{p}\in \mathrm{Spec}R$, $M_{\mathfrak{p}}$ denote the localization of $M$ at $\mathfrak{p}$.
The duality functor $\mathrm{Hom}_{R}(-,E(R/\mathfrak{m}))$ is denoted by $D(-)$, where $E(R/\mathfrak{m})$ is the injective envelope of $R/\mathfrak{m}$.

{\bf Associated prime and attached prime.} Let $M$ be an $R$-module. The set of \emph{associated prime} of $M$, denoted by $\mathrm{Ass}_{R}M$, is the set of prime ideals $\mathfrak{p}$ such that there exists a cyclic submodule $N$ of $M$ with $\mathfrak{p}=\mathrm{Ann}_{R}N$, the annihilator of $N$.

A prime ideal $\mathfrak{p}$ is said to be an \emph{attached prime} of $M$ if $\mathfrak{p}=\mathrm{Ann}_{R}(M/L)$ for some submodule $L$ of $M$. The set of attached primes of $M$ is denoted by $\mathrm{Att}_{R}M$. If $M$ is artinian, then $M$ admits a minimal secondary representation
$M=M_{1}+\cdots+M_{r}$ so that $M_{i}$ is $\mathfrak{p}_{i}$-secondary for $i=1,\cdots,r$. In this case, $\mathrm{Att}_{R}M=\{\mathfrak{p}_{1},\cdots,\mathfrak{p}_{r}\}$.

{\bf Dimension and Magnitude.} The \emph{support} of an $R$-module $M$, denoted by $\mathrm{Supp}_{R}M$, is defined as the set of prime ideals of $\mathfrak{p}$ such that there is a cyclic submodule $N$ of $M$ with $\mathrm{Ann}_{R}N\subseteq \mathfrak{p}$.
The \emph{(Krull) dimension} of $M$ is
\begin{center}
$\mathrm{dim}_{R}M=\mathrm{sup}\{\mathrm{dim}R/\mathfrak{p}\hspace{0.03cm}|\hspace{0.03cm}\mathfrak{p}\in \mathrm{Supp}_{R}M\}$.
\end{center}
If $M=0$, then write $\mathrm{dim}_{R}M=-\infty$. If $\mathrm{dim}_{R}M$ is finite, then denote $\mathrm{Assh}_{R}M=\{\mathfrak{p}\in \mathrm{Ass}_{R}M \hspace{0.03cm}|\hspace{0.03cm}\mathrm{dim}R/\mathfrak{p}=\mathrm{dim}_{R}M\}$.

Yassemi \cite{Y1} defined the cocyclic modules. An $R$-module $L$ is \emph{cocyclic} if $L$ is a submodule of $E(R/\mathfrak{m})$. The \emph{cosupport} of $M$ is defined as the set of prime ideals $\mathfrak{p}$ such that there is a cocyclic homomorphic image $L$ of $M$ with $\mathfrak{p}\supseteq \mathrm{Ann}_{R}L$, and denoted this set by $\mathrm{Cosupp}_{R}M$.
He \cite{Y2} then introduced a dual concept of dimension for modules, called it \emph{magnitude} of modules,  and is defined as
\begin{center}$\begin{aligned}
\mathrm{mag}_{R}M=\mathrm{sup}\{\mathrm{dim}R/\mathfrak{p}\hspace{0.03cm}|\hspace{0.03cm}\mathfrak{p}\in \mathrm{Cosupp}_{R}M\}.
\end{aligned}$\end{center}
If $M=0$, then write $\mathrm{mag}_{R}M=-\infty$.

Following \cite{M}, a topological $R$-module $M$ is said to be \textit{linearly topologized} if it has a base of neighborhoods of the zero element $\mathcal{M}$ consisting of submodules; $M$ is called \textit{Hausdorff} if the intersection of all the neighborhoods of the zero element is $0$. A Hausdorff linearly topologized $R$-module $M$ is said to be \textit{linearly compact} if $\mathcal{F}$ is a family of closed cosets (i.e., cosets of closed submodules) in $M$ which has the finite intersection property, then the cosets in $\mathcal{F}$ have a non-empty intersection. It should be noted that an artinian $R$-module is linearly compact. A Hausdorff linearly topologized $R$-module $M$ is called \emph{semi-discrete} if every submodule of $M$ is closed. The class of semi-discrete linearly compact modules is very large, it contains many important classes of modules such as the class of artinain modules, the class of finitely generated modules over a complete ring.

\bigskip
\section{\bf Main results }
\renewcommand{\thethm}{\arabic{section}.\arabic{thm}}

The task of this section is to give the proof of Theorems \ref{thm:A} and \ref{thm:B} in introduction. We begin with the following auxiliary lemmas.

\begin{lem}\label{lem:2.6}
Let $M$ be a finitely generated $R$-module. Suppose that $x$ is an
element in $R$ such that $\mathrm{V}(xR)\cap \mathrm{Ass}_{R}M\subseteq \{\mathfrak{m}\}$. Then $(0:_{M}x)$ has finite length.
\end{lem}

\begin{proof}
Since $(0:_{M}x)\cong \mathrm{Hom}_{R}(R/xR,M)$, $\mathrm{Ass}_{R}\mathrm{Hom}_{R}(R/xR,M)=\mathrm{V}(xR)\cap \mathrm{Ass}_{R}M\subseteq \{\mathfrak{m}\}$ and $(0:_{M}x)$ is finitely generated, we yield $(0:_{M}x)$ is artinian. Hence it has finite length.
\end{proof}

Following \cite{Z1}, an $R$-module $N$ is \emph{minimax} if there is a finitely generated submodule $L$ of $N$, such that $N/L$ is
artinian. By \cite[Lemma 2.1]{BN0}, the class of minimax $R$-modules is a Serre subcategory of the category of $R$-modules, i.e. closed under taking submodules, quotients and extensions.
Let $\mathcal{S}$ denote the subclass of linearly compact $R$-modules which consists of all finitely generated $I$-coartinian $R$-modules, and $\mathcal{S'}$ denote the subclass of linearly compact $R$-modules which consists of all minimax $I$-coartinian $R$-modules. We have

\begin{lem}\label{lem:2.2}
$\mathcal{S}$ and $\mathcal{S'}$ are both Serre subcategories of the category of $R$-modules.
\end{lem}

\begin{proof}
Let
$0\rightarrow M_{1}\rightarrow M\rightarrow M_{2}\rightarrow 0$
be a short exact sequence of linearly compact $R$-modules. Then $\mathrm{Cosupp}_{R}M=\mathrm{Cosupp}_{R}M_{1}\cup \mathrm{Cosupp}_{R}M_{2}$.
If $M_{1}$ and $M_{2}$ are finitely generated $I$-coartinian, then $M$ is finitely generated $I$-coartinian. Now assume that $M$ is finitely generated $I$-coartinian. It follows from the artinianess of $M/IM$ and the following exact sequence
$$M_{1}/IM_{1}\rightarrow M/IM\rightarrow M_{2}/IM_{2}\rightarrow 0.$$
 that $M_{2}/IM_{2}$ is artinian. Also $M_{2}$ is finitely generated, it follows from \cite[Lemma 3.5]{N1} that $M_{2}$ is $I$-coartinian, and so $M_{1}$ is $I$-coartinian. Similarly, one can show that $\mathcal{S'}$ is a Serre subcategory of the category of $R$-modules.
\end{proof}

Recall the \emph{arithmetic rank of an ideal $I$ in $R$}, denoted by $\mathrm{ara}(I)$, is the least number of elements of $R$ required to generate an ideal which has the same radical as $I$, i.e.
$$\mathrm{ara}(I)=\min \{n\in \mathbb{N}_{0}\hspace{0.03cm}|\hspace{0.03cm} \exists~ x_{1},\cdots,x_{n}\in R\ \mathrm{with}\ \mathrm{Rad}(x_{1},\cdots, x_{n})=\mathrm{Rad}(I)\}.$$
\noindent{}Let $M$ be an $R$-module. The \emph{arithmetic rank of an ideal $I$ in $R$ with respect to $M$}, denoted by $\mathrm{ara}_{M}(I)$, is defined as the arithmetic rank of the ideal
$(I+\mathrm{Ann}_{R}M)/\mathrm{Ann}_{R}M$ in the ring $R/\mathrm{Ann}_{R}M$.

\begin{lem}\label{lem:2.3}
Let $M$ be a linearly compact $R$-module and $n\geq 0$. Then the following statements are equivalent:

$\mathrm{(1)}$ For each non-zero artinian $I$-cofinite $R$-module $N$ and $0\leq i\leq n$, $\mathrm{Ext}_{R}^{i}(N,M)\in \mathcal{S}$.

$\mathrm{(2)}$ $\mathrm{Ext}_{R}^{i}(R/\mathfrak{m},M)$ is artinian for $0\leq i\leq n$.

$\mathrm{(3)}$ For each R-module $N$ of finite length and $0\leq i\leq n$, $\mathrm{Ext}_{R}^{i}(N,M)$ is artinian.
\end{lem}

\begin{proof}
$\mathrm{(1)}\Rightarrow \mathrm{(2)}$ The assertion is clear as
$R/I\otimes_{R} \mathrm{Ext}_{R}^{i}(R/\mathfrak{m},M)\cong \mathrm{Ext}_{R}^{i}(R/\mathfrak{m},M)$.

$\mathrm{(2)}\Rightarrow \mathrm{(3)}$ Since $N$ has finite length, there is a finite filtration
$$0=N_{0}\subseteq N_{1}\subseteq \cdots\subseteq N_{k}=N,$$
such that $N_{i}/N_{i-1}\cong R/\mathfrak{m}$ for $0\leq i\leq k$. Now, using induction on $k$ and with the exact sequence
$$0\rightarrow N_{i-1}\rightarrow N_{i}\rightarrow N_{i}/N_{i-1}\rightarrow 0$$
for $0\leq i\leq k$, the result is easy to prove.

$\mathrm{(3)}\Rightarrow \mathrm{(1)}$ In view of \cite[Proposition 4.6]{N}, we can assume that $R$ is complete. Since $\mathrm{Ext}_{R}^{i}(N,M)$ is linearly compact by \cite[Lemma 2.5]{CN1}, we only need to show $\mathrm{Ext}_{R}^{i}(N,M)$ is finitely generated $I$-coartinian for $0\leq i\leq n$. Suppose that $N$ is artinian $I$-cofinite. Then $\mathrm{Supp}_{R}N\subseteq \{\mathfrak{m}\}$. In order to prove the assertion, we use induction on
$$t=\mathrm{ara}_{N}(I)=\mathrm{ara}(I+\mathrm{Ann}_{R}N)/\mathrm{Ann}_{R}N.$$
If $t=0$, then $I^{u}N=0$ for some $u>0$. As
$\mathrm{Hom}_{R}(R/I,N)$ is finitely generated, it follows that $N=\mathrm{Hom}_{R}(R/I^{u},N)$ is finitely generated. Therefore, $N$ has finite length and $\mathrm{Ext}_{R}^{i}(N,M)$ is artinian for $0\leq i\leq n$. Also by \cite[corollary 3.6]{N} and \cite[Lemma 2.5]{Y1}, we have
\begin{center}$\begin{aligned}
\mathrm{Supp}_{R}D(\mathrm{Ext}_{R}^{i}(N,M))
&= \mathrm{Cosupp}_{R}\mathrm{Ext}_{R}^{i}(N,M) \\
&\subseteq
\mathrm{Supp}_{R}N \cap \mathrm{Cosupp}_{R}M \\
&\subseteq
\mathrm{Supp}_{R}N
\subseteq
\{\mathfrak{m}\}.
\end{aligned}$\end{center}
Since $D(\mathrm{Ext}_{R}^{i}(N,M))$ is finitely generated, $D(\mathrm{Ext}_{R}^{i}(N,M))$ is artinian $I$-cofinite by \cite[Lemma 2.1]{M6}. Hence $\mathrm{Ext}_{R}^{i}(N,M)$ is finitely generated $I$-coartinian for $0\leq i\leq n$. Suppose, inductively, that $t>0$ and the result has been proved for smaller values of $t$. We can conclude that $N$ is not of finite length and so $\mathrm{Att}_{R}N\nsubseteq \{\mathfrak{m}\}$. By hypothesis $N$ is artinian $I$-cofinite, one has $\mathrm{Att}_{R}N \cap \mathrm{V}(I)\subseteq \{\mathfrak{m}\}$ by \cite[Lemma 2.3]{B1}. On the other hand, there exist elements $y_{1},\cdots,y_{t}\in I$ such that
$$\mathrm{Rad}(I+\mathrm{Ann}_{R}N/\mathrm{Ann}_{R}N)=\mathrm{Rad}((y_{1},\cdots,y_{t})+\mathrm{Ann}_{R}N/\mathrm{Ann}_{R}N).$$
By Prime Avoidance Theorem,
$$I\nsubseteq \bigcup_{\mathfrak{q}\in \mathrm{Att}_{R}N\backslash \mathrm{V}(I)}\mathfrak{q},$$
which implies that
$$(y_{1},\cdots,y_{t})+\mathrm{Ann}_{R}N \nsubseteq \bigcup_{\mathfrak{q}\in \mathrm{Att}_{R}N\backslash \mathrm{V}(I)}\mathfrak{q}.$$
But
$\mathrm{Ann}_{R}N\subseteq \mathop{\bigcap}_{\mathfrak{q}\in \mathrm{Att}_{R}N\backslash \mathrm{V}(I)}\mathfrak{q}$,
consequently,
$$(y_{1},\cdots,y_{t})\nsubseteq \bigcup_{\mathfrak{q}\in \mathrm{Att}_{R}N\backslash \mathrm{V}(I)}\mathfrak{q}.$$
By \cite[Theorem 16.8]{M5} there is $a\in (y_{2},\cdots,y_{t})$ such that
$y_{1}+a\notin \bigcup_{\mathfrak{q}\in \mathrm{Att}_{R}N\backslash \mathrm{V}(I)}\mathfrak{q}$. Let $x=y_{1}+a$. Then $x\in I$,
$\mathrm{Att}_{R}N\cap \mathrm{V}(xR)\subseteq \mathrm{Att}_{R}N \cap \mathrm{V}(I)\subseteq \{\mathfrak{m}\}$
and
$$\mathrm{Rad}(I+\mathrm{Ann}_{R}N/\mathrm{Ann}_{R}N)=\mathrm{Rad}((x,y_{2},\cdots,y_{t})+\mathrm{Ann}_{R}N/\mathrm{Ann}_{R}N).$$
As $N$ is artinian, it can be seen that $x^{k}N=x^{k+1}N$ for some $k>0$. Put $B=x^{k}N$ and $C=(0:_{B}x)$. Then
$$\mathrm{Rad}(I+\mathrm{Ann}_{R}C/\mathrm{Ann}_{R}C)=\mathrm{Rad}((y_{2},\cdots,y_{t})+\mathrm{Ann}_{R}C/\mathrm{Ann}_{R}C),$$
and hence $\mathrm{ara}(I+\mathrm{Ann}_{R}C)/\mathrm{Ann}_{R}C\leq t-1$. Since
$$\mathrm{Att}_{R}N\cap \mathrm{V}(x^{k}R)= \mathrm{Att}_{R}N\cap \mathrm{V}(xR)\subseteq \{\mathfrak{m}\},$$
it follows from \cite[Lemma 2.4]{BN} that $N/B=N/x^{k}N$ has finite length. In
view of condition, $\mathrm{Ext}_{R}^{i}(N/B,M)\in \mathcal{S}$ for
$0\leq i\leq n$. Note that $B$ and $C$ are both artinian $I$-cofinite and there is a short exact sequence
$$\xymatrix{0\ar[r] & C\ar[r] & B\ar[r]^{x} & B\ar[r] & 0.}$$
By \cite[Lemma 2.5]{CN1}, we have the long exact sequence of linearly compact $R$-modules
$$\xymatrix@C=0.5cm{\cdots\ar[r] & \mathrm{Ext}_{R}^{i-1}(C,M)\ar[r] & \mathrm{Ext}_{R}^{i}(B,M)\ar[r]^{x} &  \mathrm{Ext}_{R}^{i}(B,M)\ar[r] & \mathrm{Ext}_{R}^{i}(C,M)\ar[r] & }$$
$$\xymatrix@C=0.5cm{\mathrm{Ext}_{R}^{i+1}(B,M)\ar[r]^{x} & \mathrm{Ext}_{R}^{i+1}(B,M)\ar[r] & \cdots.}$$
By inductive hypothesis, $\mathrm{Ext}_{R}^{i}(C,M)\in \mathcal{S}$ for $0 \leq i\leq n$, it follows from Lemma \ref{lem:2.2} that $(0:_{\mathrm{Ext}_{R}^{i}(B,M)}x)\in \mathcal{S}$ for $0 \leq i\leq n+1$ and $\mathrm{Ext}_{R}^{i}(B,M)/x\mathrm{Ext}_{R}^{i}(B,M)\in \mathcal{S}$ for $0 \leq i\leq n$. Hence $\mathrm{Ext}_{R}^{i}(B,M)$ is $I$-coartinian by \cite[Proposition 2.6]{S2}, and so $\mathrm{Ext}_{R}^{i}(B,M)\in \mathcal{S}$ for $0\leq i\leq n$. Also the exact sequence
$$0\rightarrow B\rightarrow N\rightarrow N/B\rightarrow 0$$
yields the long exact sequence of linearly compact $R$-modules
$$\xymatrix@C=0.5cm{\cdots\ar[r] & \mathrm{Ext}_{R}^{i}(N/B,M)\ar[r] & \mathrm{Ext}_{R}^{i}(N,M)\ar[r] & \mathrm{Ext}_{R}^{i}(B,M)\ar[r] & \cdots,}$$
which implies that $\mathrm{Ext}_{R}^{i}(N,M)\in \mathcal{S}$ for $0\leq i\leq n$. This completes the inductive step.
\end{proof}

\begin{cor}\label{cor:2.4}
Let $M$ be a linearly compact $R$-module and $n\geq0$. Then the following statements are equivalent:

$\mathrm{(1)}$ For each minimax $I$-cofinite $R$-module $N$ and $0\leq i\leq n$, $\mathrm{Ext}_{R}^{i}(N,M)$ is minimax $I$-coartinian.

$\mathrm{(2)}$ For $0\leq i\leq n$, $\mathrm{Ext}_{R}^{i}(R/I,M)$ is artinian.
\end{cor}

\begin{proof}
$\mathrm{(1)}\Rightarrow \mathrm{(2)}$ Using the fact that $I\mathrm{Ext}_{R}^{i}(R/I,M)=0$ for
$0\leq i\leq n$.

$\mathrm{(2)}\Rightarrow \mathrm{(1)}$ Let $N$ be a minimax $I$-cofinite $R$-module. There exists a finitely generated submodule $L$ of $N$, such that $N/L$ is artinian. By \cite[Corollary 4.4]{M2}, $N/L$ is $I$-cofinite. By hypothesis and \cite[Lemma 2.2]{S2}, $\mathrm{Ext}_{R}^{i}(R/\mathfrak{m},M)$ is artinian for $0\leq i\leq n$. Hence Lemma \ref{lem:2.3} implies that $\mathrm{Ext}_{R}^{i}(N/L,M)\in \mathcal{S}$ for $0\leq i\leq n$. Also by \cite[Lemma 2.2]{S2}, $\mathrm{Ext}_{R}^{i}(L,M)$ is artinian for $0\leq i\leq n$, and note that
$$\mathrm{Cosupp}_{R}\mathrm{Ext}_{R}^{i}(L,M)\subseteq \mathrm{Supp}_{R}L\cap \mathrm{Cosupp}_{R}M\subseteq \mathrm{V}(I).$$
Now, the induced long exact sequence
$$\cdots\rightarrow \mathrm{Ext}_{R}^{i}(L,M)\rightarrow \mathrm{Ext}_{R}^{i}(N,M)\rightarrow \mathrm{Ext}_{R}^{i}(N/L,M)\rightarrow \cdots$$
implies that $\mathrm{Ext}_{R}^{i}(N,M)$ is minimax $I$-coartinian for $0\leq i\leq n$.
\end{proof}

We now give the proof of the first theorem.

\begin{proof}[Proof of Theorem \ref{thm:A}]
Since $\widehat{R}$ is a faithful flat $R$-module and $N\in \mathcal{C}^{1}(R,I)_{cof}$, we have $N\otimes_{R}\widehat{R}\in \mathcal{C}^{1}(\widehat{R},I\widehat{R})_{cof}$. In view of $\mathrm{Ext}_{R}^{i}(N,M)$ is a linearly compact $\widehat{R}$-module,
$\mathrm{Ext}_{\widehat{R}}^{i}(N\otimes_{R}\widehat{R}, M)\cong \mathrm{Ext}_{R}^{i}(N,M)$ and \cite[Proposition 4.6]{N}, we can assume that $R$ is complete. Since  $\mathrm{Cosupp}_{R}\mathrm{Ext}_{R}^{i}(N,M)\subseteq \mathrm{Supp}_{R}N \cap \mathrm{Cosupp}_{R}M \subseteq \mathrm{Supp}_{R}N$ by \cite[Corollary 3.6]{N},
$\mathrm{mag}_{R}\mathrm{Ext}_{R}^{i}(N,M)\leq 1$. In order to prove the assertion, we use
induction on
$$t=\mathrm{ara}_{N}(I)=\mathrm{ara}(I+\mathrm{Ann}_{R}N)/\mathrm{Ann}_{R}N.$$
If $t=0$, then $\mathrm{Hom}_{R}(R/I^{l},N)=N$ is finitely generated for some integer $l$, it follows from \cite[Proposition 3.3]{B} and \cite[Lemma 2.2]{S2} that $\mathrm{Ext}_{R}^{i}(N,M)$ is artinian for $i\geq 0$. Hence
$\mathrm{Ext}_{R}^{i}(N,M) \in \mathcal{C}^{1}(R,I)_{coa}$ for $i\geq 0$. Suppose, inductively, that $t>0$ and the result has been proved for smaller values of $t$. Since $N/IN$ is finitely generated, there is a finitely generated submodule $N_{1}$ of $N$ such that
$N/IN=(N_{1}+IN)/IN$,
one has $N=N_{1}+IN$. Since $N_{1}$ is an $I$-torsion finitely generated $R$-module, it follows that
$I^{k}N_{1}=0$ for some $k>0$. Therefore, $I^{k}N=I^{k}N_{1}+I^{k+1}N=I^{k+1}N$. Set $L=I^{k}N$. Then $IL=L\in \mathcal{C}^{1}(R,I)_{cof}$, $T=N/L$ is finitely generated and
$$\mathrm{ara}(I+\mathrm{Ann}_{R}L)/\mathrm{Ann}_{R}L\leq \mathrm{ara}(I+\mathrm{Ann}_{R}N)/\mathrm{Ann}_{R}N.$$
Also the short exact sequence
$$0\rightarrow L\rightarrow N\rightarrow T\rightarrow 0$$
yields the exact sequence
$$\cdots\rightarrow \mathrm{Ext}_{R}^{i-1}(L,M)\rightarrow \mathrm{Ext}_{R}^{i}(T,M)\rightarrow \mathrm{Ext}_{R}^{i}(N,M)\rightarrow \mathrm{Ext}_{R}^{i}(L,M)\rightarrow \mathrm{Ext}_{R}^{i+1}(T,M)\rightarrow \cdots.$$
Since $T$ is finitely generated and
$\mathrm{Supp}_{R}T\subseteq \mathrm{Supp}_{R}N\subseteq \mathrm{V}(I)$, $\mathrm{Ext}_{R}^{i}(T,M)$ is artinian by \cite[Lemma 2.2]{S2} for $i\geq 0$. Note that $\mathrm{Cosupp}_{R}\mathrm{Ext}_{R}^{i}(T,M)\subseteq \mathrm{Supp}_{R}T\cap \mathrm{Cosupp}_{R}M\subseteq \mathrm{V}(I)$
by \cite[Corollary 3.6]{N}, thus $\mathrm{Ext}_{R}^{i}(T,M)\in \mathcal{C}^{1}(R,I)_{coa}$ for
$i\geq 0$. If the assertion is true for $L$, then
$\mathrm{Ext}_{R}^{i}(L,M)\in \mathcal{C}^{1}(R,I)_{coa}$ for $i\geq 0$.
Moreover, $\mathcal{C}^{1}(R,I)_{coa}$ is an abelian category by \cite[Corollary 3.10]{S2}. By using the last exact
sequence we can see that the assertion also is true for $N$. Thus replacing $N$ by $L$, without loss of generality, we may assume that $IN=N$. If $N$ is minimax, then the assertion follows from Corollary \ref{cor:2.4}. So we may assume that $N$ is not minimax. As $\mathrm{Hom}_{R}(R/I,N)$ is finitely generated and $\mathrm{Supp}_{R}N\subseteq \mathrm{V}(I)$, it follows from \cite[Lemma 2.1]{M6} that $\mathrm{Supp}_{R}N\nsubseteq \{\mathfrak{m}\}$, and hence $\mathrm{dim}_{R}N=1$. Let
$$\mathrm{Assh}_{R}N=\{\mathfrak{p}_{1},\ldots,\mathfrak{p}_{u}\}.$$
By \cite[Lemma 2.1]{M6}, it can be seen that for $1\leq j\leq u$, $N_{\mathfrak{p}_{j}}$ is an artinian and $IR_{\mathfrak{p}_{j}}$-cofinite $R_{\mathfrak{p}_{j}}$-module with
$N_{\mathfrak{p}_{j}}=(IR_{\mathfrak{p}_{j}})N_{\mathfrak{p}_{j}}$. Thus
$IR_{\mathfrak{p}_{j}}\nsubseteq \bigcup_{\mathfrak{q}R_{\mathfrak{p}_{j}}\in \mathrm{Att}_{R_{\mathfrak{p}_{j}}}N_{\mathfrak{p}_{j}}}\mathfrak{q}R_{\mathfrak{p}_{j}}$.
Let
$$\mathcal{U}_{1}:=\mathop{\bigcup^{u}}_{j=1}\{\mathfrak{q}\in \mathrm{Spec}R\mid \mathfrak{q}R_{\mathfrak{p}_{j}}\in \mathrm{Att}_{R_{\mathfrak{p}_{j}}}N_{\mathfrak{p}_{j}}\}.$$
By Prime Avoidance Theorem,
$$I\nsubseteq \bigcup_{\mathfrak{q}\in \mathcal{U}_{1}}\mathfrak{q}.$$
For each $\mathfrak{q}\in \mathcal{U}_{1}$, one has
$\mathfrak{q}R_{\mathfrak{p}_{j}}\in \mathrm{Att}_{R_{\mathfrak{p}_{j}}}N_{\mathfrak{p}_{j}}$ for
 some $1\leq j\leq u$, so
$$(\mathrm{Ann}_{R}N)R_{\mathfrak{p}_{j}}\subseteq \mathrm{Ann}_{R_{\mathfrak{p}_{j}}}N_{\mathfrak{p}_{j}}\subseteq \mathfrak{q}R_{\mathfrak{p}_{j}},$$
which implies that $\mathrm{Ann}_{R}N\subseteq \mathfrak{q}$. Therefore, $\mathcal{U}_{1}\subseteq \mathrm{V}(\mathrm{Ann}_{R}N)$. Next, let
$$\mathcal{T}:=\mathop{\bigcup^{n+1}}_{i=0}\{\mathfrak{p}\in \mathrm{Cosupp}_{R}\mathrm{Ext}_{R}^{i}(N,M)\mid \mathrm{dim}R/\mathfrak{p}=1\}.$$
As $N\in \mathcal{C}^{1}(R,I)_{cof}$, the set $\mathrm{Ass}_{R}N$ is finite, and hence $\mathcal{T}\subseteq \mathrm{Assh}_{R}N$ is finite. Because $\mathrm{Cosupp}_{R_{\mathfrak{p}}}\mathrm{Hom}_{R}(R_{\mathfrak{p}},\mathrm{Ext}_{R}^{i}(N,M))\subseteq \mathrm{V}(\mathfrak{p}R_{\mathfrak{p}})$ for $\mathfrak{p}\in \mathcal{T}$, $\mathrm{mag}_{R_{\mathfrak{p}}}\mathrm{Hom}_{R}(R_{\mathfrak{p}},\mathrm{Ext}_{R}^{i}(N,M))=0$. Since $\mathrm{Ext}_{R}^{i}(N,M)$ is linearly compact, it follows from \cite[5.5]{M} that $$\mathrm{Ext}_{R}^{i}(N,M)\cong \underleftarrow{\text{lim}}_{M'\in \mathcal{M}}\mathrm{Ext}_{R}^{i}(N,M)/M',$$ where $\mathrm{Ext}_{R}^{i}(N,M)/M'$ is an artinian $R$-module and $\mathcal{M}$ a base consisting of neighborhood of the zero element of $\mathrm{Ext}_{R}^{i}(N,M)$.
Since $\mathrm{Hom}_{R}(R_{\mathfrak{p}},\mathrm{Ext}_{R}^{i}(N,M)/M')$ is a semi-discrete linearly compact $R$-module, there is a short exact sequence
$$0\rightarrow K\rightarrow \mathrm{Hom}_{R}(R_{\mathfrak{p}},\mathrm{Ext}_{R}^{i}(N,M)/M')\rightarrow A\rightarrow 0,$$
where $K$ is finitely generated and $A$ is artinian. Then $\mathrm{Hom}_{R}(R_{\mathfrak{p}},K)=0$, and therefore $\mathrm{Hom}_{R}(R_{\mathfrak{p}},\mathrm{Ext}_{R}^{i}(N,M)/M')\cong \mathrm{Hom}_{R}(R_{\mathfrak{p}},A)$. By \cite[Theorem 3.2]{M0}, $\mathrm{Hom}_{R}(R_{\mathfrak{p}},A)$ is annihilated by $\mathfrak{p}R_{\mathfrak{p}}$, so $\mathrm{Hom}_{R}(R_{\mathfrak{p}},A)$ is artinian by \cite[Theorem 7.30]{S0}. Hence $R_{\mathfrak{p}}/IR_{\mathfrak{p}}\otimes_{R_{\mathfrak{p}}}\mathrm{Hom}_{R}(R_{\mathfrak{p}},\mathrm{Ext}_{R}^{i}(N,M)/M')$
is artinain. Let
$$\mathcal{T}=\{\mathfrak{p}_{1},\cdots,\mathfrak{p}_{s}\}, s\leq u.$$
By \cite[Lemma 3.7]{S2}, we have
$$\mathrm{V}(IR_{\mathfrak{p}_{j}})\cap \mathrm{Ass}_{R_{\mathfrak{p}_{j}}}\mathrm{Hom}_{R}(R_{\mathfrak{p}_{j}},\mathrm{Ext}_{R}^{i}(N,M)/M')\subseteq \mathrm{V}(\mathfrak{p}_{j}R_{\mathfrak{p}_{j}})$$
for $j=1,\cdots,s$. If $\mathrm{V}(IR_{\mathfrak{p}_{j}})=\{\mathfrak{p}_{j}R_{\mathfrak{p}_{j}}\}$, then
$$\mathrm{V}(IR_{\mathfrak{p}_{j}})\cap \mathrm{Ass}_{R_{\mathfrak{p}_{j}}}\mathrm{Hom}_{R}(R_{\mathfrak{p}_{j}},\mathrm{Ext}_{R}^{i}(N,M))\subseteq \mathrm{V}(\mathfrak{p}_{j}R_{\mathfrak{p}_{j}})$$
for $j=1,\cdots,s$. Otherwise, for every $\mathfrak{q}\in \mathrm{Spec}R$ with $IR_{\mathfrak{p}_{j}}\subseteq \mathfrak{q}R_{\mathfrak{p}_{j}}\subsetneq  \mathfrak{p_{j}}R_{\mathfrak{p}_{j}}$, we have
$$\mathrm{Hom}_{R_{\mathfrak{p}_{j}}}(R_{\mathfrak{p}_{j}}/\mathfrak{q}R_{\mathfrak{p}_{j}},\mathrm{Ext}_{R}^{i}(N,M)/M')=0.$$
Therefore,
\begin{center}$\begin{aligned}
\mathrm{Hom}_{R_{\mathfrak{p}_{j}}}(R_{\mathfrak{p}_{j}}/\mathfrak{q}R_{\mathfrak{p}_{j}},\mathrm{Ext}_{R}^{i}(N,M))
&= \mathrm{Hom}_{R_{\mathfrak{p}_{j}}}(R_{\mathfrak{p}_{j}}/\mathfrak{q}R_{\mathfrak{p}_{j}},\underleftarrow{\text{lim}}_{M'\in \mathcal{M}}\mathrm{Ext}_{R}^{i}(N,M)/M') \\
&=
\underleftarrow{\text{lim}}_{M'\in \mathcal{M}}\mathrm{Hom}_{R_{\mathfrak{p}_{j}}}(R_{\mathfrak{p}_{j}}/\mathfrak{q}R_{\mathfrak{p}_{j}},\mathrm{Ext}_{R}^{i}(N,M)/M')=0,
\end{aligned}$\end{center}
which implies that
$$\mathrm{V}(IR_{\mathfrak{p}_{j}})\cap \mathrm{Ass}_{R_{\mathfrak{p}_{j}}}\mathrm{Hom}_{R}(R_{\mathfrak{p}_{j}},\mathrm{Ext}_{R}^{i}(N,M))\subseteq \mathrm{V}(\mathfrak{p}_{j}R_{\mathfrak{p}_{j}})$$
for $j=1,\cdots,s$.
Next, let
$$\mathcal{U}_{2}:=\mathop{\bigcup^{n+1}}_{i=0}\mathop{\bigcup^{s}}_{j=1}\{\mathfrak{q}\in \mathrm{Spec}R\mid \mathfrak{q}R_{\mathfrak{p}_{j}}\in \mathrm{Ass}_{R_{\mathfrak{p}_{j}}}\mathrm{Hom}_{R}(R_{\mathfrak{p}_{j}},\mathrm{Ext}_{R}^{i}(N,M))\}.$$
It is easy to see that $\mathcal{U}_{2}\cap \mathrm{V}(I)\subseteq \mathcal{T}$. Also for each $\mathfrak{q}\in \mathcal{U}_{2}$, we have
$$\mathfrak{q}R_{\mathfrak{p}_{j}}\in \mathrm{Ass}_{R_{\mathfrak{p}_{j}}}\mathrm{Hom}_{R}(R_{\mathfrak{p}_{j}},\mathrm{Ext}_{R}^{i}(N,M))$$
for some $0\leq i\leq n+1$ and $1\leq j\leq s$. Hence
$$(\mathrm{Ann}_{R}N)R_{\mathfrak{p}_{j}}\subseteq \mathrm{Ann}_{R}(\mathrm{Ext}_{R}^{i}(N,M))R_{\mathfrak{p}_{j}}\subseteq \mathrm{Ann}_{R_{\mathfrak{p}_{j}}}\mathrm{Hom}_{R}(R_{\mathfrak{p}_{j}},\mathrm{Ext}_{R}^{i}(N,M))\subseteq \mathfrak{q}R_{\mathfrak{p}_{j}},$$
and so $\mathrm{Ann}_{R}N\subseteq \mathfrak{q}$. Therefore,
$\mathcal{U}_{2}\subseteq \mathrm{V}(\mathrm{Ann}_{R}N)$. Set
$\mathcal{U}= \mathcal{U}_{1}\cup \mathcal{U}_{2}$. Since
$t=\mathrm{ara}_{N}(I)\geq 1$, there exists elements $y_{1},\cdots,y_{t}\in I$ such that
$$\mathrm{Rad}(I+\mathrm{Ann}_{R}N/\mathrm{Ann}_{R}N)=\mathrm{Rad}((y_{1},\cdots,y_{t})+\mathrm{Ann}_{R}N/\mathrm{Ann}_{R}N).$$
By Prime Avoidance Theorem,
$$I\nsubseteq \bigcup_{\mathfrak{q}\in \mathcal{U}\backslash \mathrm{V}(I)}\mathfrak{q},$$
which implies that
$$(y_{1},\cdots,y_{t})+\mathrm{Ann}_{R}N \nsubseteq \bigcup_{\mathfrak{q}\in \mathcal{U}\backslash \mathrm{V}(I)}\mathfrak{q}.$$
Note that
$\mathrm{Ann}_{R}N\subseteq \mathop{\bigcap}_{\mathfrak{q}\in \mathcal{U}\backslash \mathrm{V}(I)}\mathfrak{q}$,
so
$$(y_{1},\cdots,y_{t})\nsubseteq \bigcup_{\mathfrak{q}\in \mathcal{U}\backslash \mathrm{V}(I)}\mathfrak{q}.$$
By \cite[Theorem 16.8]{M5} there is $a\in (y_{2},\cdots,y_{t})$ such that
$y_{1}+a\notin \bigcup_{\mathfrak{q}\in \mathcal{U}\backslash \mathrm{V}(I)}\mathfrak{q}$. Let $x=y_{1}+a$. Then $x\in I$ and
$$\mathrm{Rad}(I+\mathrm{Ann}_{R}N/\mathrm{Ann}_{R}N)=\mathrm{Rad}((x,y_{2},\cdots,y_{t})+\mathrm{Ann}_{R}N/\mathrm{Ann}_{R}N).$$
Next, let $U=(0:_{N}x)$, $V=N/xN$. Then
$$\mathrm{ara}_{U}(I)=\mathrm{ara}(I+\mathrm{Ann}_{R}U)/\mathrm{Ann}_{R}U\leq t-1.$$
As $N\in \mathcal{C}^{1}(R,I)_{cof}$ and $\mathcal{C}^{1}(R,I)_{cof}$ is an abelian category, it follows that $U$, $V\in \mathcal{C}^{1}(R,I)_{cof}$. By the inductive hypothesis, $\mathrm{Ext}_{R}^{i}(U,M)\in \mathcal{C}^{1}(R,I)_{coa}$ for $0\leq i \leq n$. Also $V_{\mathfrak{p}_{j}}=0$ for $1\leq j\leq u$, so
$\mathrm{Supp}_{R}V\subseteq \mathrm{Supp}_{R}N\backslash \{\mathfrak{p}_{1},\ldots,\mathfrak{p}_{u}\}\subseteq \{\mathfrak{m}\}$, and hence $V$ is artinian $I$-cofinite by \cite[Lemma 2.1]{M6}. By the hypothesis and \cite[Lemma 2.2]{S2}, one sees that $\mathrm{Ext}_{R}^{i}(R/\mathfrak{m},M)$ is artinain for $0\leq i\leq n+1$, it follows from Lemma \ref{lem:2.3} that the $R$-module $\mathrm{Ext}_{R}^{i}(V,M)\in \mathcal{S}$ for $0\leq i\leq n+1$. Also the short exact sequences
$$\xymatrix{0\ar[r] & U\ar[r] & N\ar[r]^{f} & xN\ar[r] & 0},$$
$$\xymatrix{0\ar[r] & xN\ar[r]^{g} & N\ar[r] & V\ar[r] & 0}$$
induce the exact sequences
$$\xymatrix{\mathrm{Ext}_{R}^{i-1}(U,M)\ar[r] & \mathrm{Ext}_{R}^{i}(xN,M)\ar[r]^{\mathrm{Ext}_{R}^{i}(f,M)} & \mathrm{Ext}_{R}^{i}(N,M)\ar[r] & \mathrm{Ext}_{R}^{i}(U,M),} \eqno(2.5.1)$$
$$\xymatrix{\mathrm{Ext}_{R}^{i}(V,M)\ar[r] & \mathrm{Ext}_{R}^{i}(N,M)\ar[r]^{\mathrm{Ext}_{R}^{i}(g,M)} & \mathrm{Ext}_{R}^{i}(xN,M)\ar[r] & \mathrm{Ext}_{R}^{i+1}(V,M).}\eqno(2.5.2)$$
From $(2.5.2)$ and Lemma \ref{lem:2.2} we can deduce
that $\mathrm{ker}\mathrm{Ext}_{R}^{i}(g,M)$ and $\mathrm{coker}\mathrm{Ext}_{R}^{i}(g,M)$ are in $\mathcal{S}$ for $0\leq i\leq n$. From $(2.5.1)$ one has
$R/I \otimes_{R} \mathrm{ker}\mathrm{Ext}_{R}^{i}(f,M)$ is artinian for $0\leq i\leq n+1$.  Hence the exact sequence
$$\xymatrix{0\ar[r] & \mathrm{kerExt}_{R}^{i}(g,M)\ar[r] & \mathrm{kerExt}_{R}^{i}(gf,M)\ar[r]^-{\delta^{i}} & \mathrm{kerExt}_{R}^{i}(f,M)\ar[r] & W^{i}\ar[r] & 0,}$$
and the monmorphism $W^{i}\rightarrow \mathrm{cokerExt}_{R}^{i}(g,M)$ imply that $W^{i}\in \mathcal{S}$ for $0\leq i \leq n$. Since $\mathrm{Ext}_{R}^{i}(gf,M)=\mathrm{Ext}_{R}^{i}(x,M)$, we have $\mathrm{kerExt}_{R}^{i}(gf,M)\cong (0:_{\mathrm{Ext}_{R}^{i}(N,M)}x)$. At this point, we show that $Y^{i}:=(0:_{\mathrm{Ext}_{R}^{i}(N,M)}x)$ is minimax for $0\leq i \leq n$. Since
\begin{center}$\begin{aligned}
\mathrm{V}(xR_{\mathfrak{p}_{j}})\cap \mathrm{Ass}_{R_{\mathfrak{p}_{j}}}\mathrm{Hom}_{R}(R_{\mathfrak{p}_{j}},Y^{i})
&\subseteq \mathrm{V}(IR_{\mathfrak{p}_{j}})\cap \mathrm{Ass}_{R_{\mathfrak{p}_{j}}}\mathrm{Hom}_{R}(R_{\mathfrak{p}_{j}},Y^{i}) \\
&\subseteq \mathrm{V}(IR_{\mathfrak{p}_{j}})\cap \mathrm{Ass}_{R_{\mathfrak{p}_{j}}}\mathrm{Hom}_{R}(R_{\mathfrak{p}_{j}},\mathrm{Ext}_{R}^{i}(N,M)) \\
&\subseteq \mathrm{V}(\mathfrak{p}_{j}R_{\mathfrak{p}_{j}}),
\end{aligned}$\end{center}it follows from Lemma \ref{lem:2.6} that
 the $R_{\mathfrak{p}_{j}}$-module
\begin{center}$
\mathrm{Hom}_{R_{\mathfrak{p}_{j}}}(R_{\mathfrak{p}_{j}}/xR_{\mathfrak{p}_{j}},\mathrm{Hom}_{R}(R_{\mathfrak{p}_{j}},Y^{i}))
\cong \mathrm{Hom}_{R}(R_{\mathfrak{p}_{j}}\otimes_{R}R/xR,Y^{i})
\cong \mathrm{Hom}_{R}(R_{\mathfrak{p}_{j}},Y^{i})$\end{center}
is of finite length for $1\leq j\leq s$ and $0\leq i\leq n$, so there exists an artinian quotient module $\widetilde{Y}^{i,j}$ of $Y^{i}$ such that $\mathrm{Hom}_{R}(R_{\mathfrak{p}_{j}},\widetilde{Y}^{i,j})=\mathrm{Hom}_{R}(R_{\mathfrak{p}_{j}},Y^{i})$. Set $\widetilde{Y}^{i}=\widetilde{Y}^{i,1}+\cdots+\widetilde{Y}^{i,s}$ for $0\leq i\leq n$. Then $\widetilde{Y}^{i}$ is an artinian quotient module of $Y^{i}$ so that $\mathrm{Hom}_{R}(R_{\mathfrak{p}_{j}},\widetilde{Y}^{i})=\mathrm{Hom}_{R}(R_{\mathfrak{p}_{j}},Y^{i})$ for $1\leq j\leq s$. Thus the exact sequence
$$0\rightarrow K^{i}\rightarrow Y^{i}\rightarrow \widetilde{Y}^{i}\rightarrow 0\eqno(2.5.3)$$
 implies that $\mathrm{Hom}_{R}(R_{\mathfrak{p}_{j}},K^{i})=0$, and hence $\mathrm{Supp}_{R}D(K^{i})=\mathrm{Cosupp}_{R}K^{i}\subseteq \{\mathfrak{m}\}$. Also the exact sequences
$$0\rightarrow \mathrm{im}\delta^{i}\rightarrow \mathrm{kerExt}_{R}^{i}(f,M)\rightarrow W^{i}\rightarrow 0,$$
$$0\rightarrow \mathrm{kerExt}_{R}^{i}(g,M)\rightarrow \mathrm{kerExt}_{R}^{i}(gf,M)\rightarrow \mathrm{im}\delta^{i}\rightarrow 0$$
induce the exact sequences
$$\mathrm{Tor}_{i}^{R}(R/I,W^{i})\rightarrow R/I\otimes_{R}\mathrm{im}\delta^{i}\rightarrow R/I\otimes_{R}\mathrm{kerExt}_{R}^{i}(f,M),$$
$$R/I\otimes_{R}\mathrm{kerExt}_{R}^{i}(g,M)\rightarrow R/I\otimes_{R}\mathrm{kerExt}_{R}^{i}(gf,M)\rightarrow R/I\otimes_{R}\mathrm{im}\delta^{i}\rightarrow 0,$$
which implies that
$R/I\otimes_{R}\mathrm{kerExt}_{R}^{i}(gf,M)=Y^{i}/IY^{i}$
is artinian for $0\leq i \leq n$. The induced exact sequence
$$\mathrm{Tor}_{1}^{R}(R/I, \widetilde{Y}^{i})\rightarrow K^{i}/IK^{i}\rightarrow Y^{i}/IY^{i}$$
shows that $\mathrm{Hom}_{R}(R/I,D(K^{i}))\cong D(K^{i}/IK^{i})$ is finitely generated. Hence $D(K^{i})$ is artinian by \cite[Lemma 2.1]{M6}, and so $K^{i}$ is finitely generated. Therefore, the $R$-module $Y^{i}$ is minimax for $0\leq i\leq n$. Furthermore, $Y^{i}$ is $I$-coartinian for $0\leq i\leq n$ by \cite[Lemma 3.5]{N1}. At the next step we show that $\mathrm{kerExt}_{R}^{i}(f,M)$ is minimax $I$-coartinian for $0 \leq i \leq n+1$.
Since
\begin{center}$\begin{aligned}
\mathrm{Supp}_{R}D(\mathrm{kerExt}_{R}^{i}(f,M))
&= \mathrm{Cosupp}_{R}\mathrm{kerExt}_{R}^{i}(f,M)\\
&= \mathrm{Cosupp}_{R}\mathrm{im}\delta^{i}\cup \mathrm{Cosupp}_{R}W^{i} \\
&\subseteq
\mathrm{Cosupp}_{R}\mathrm{kerExt}_{R}^{i}(gf,M)\cup (\mathrm{V}(I)\cap \{\mathfrak{m}\}) \\
&\subseteq
\mathrm{V}(I)\cup (\mathrm{V}(I)\cap \{\mathfrak{m}\})
=
\mathrm{V}(I)\cap \{\mathfrak{m}\},
\end{aligned}$\end{center}
and $\mathrm{Hom}_{R}(R/I,D(\mathrm{kerExt}_{R}^{i}(f,M)))\cong D(R/I \otimes_{R} \mathrm{kerExt}_{R}^{i}(f,M))$ is finitely generated, we have the $R$-module $D(\mathrm{kerExt}_{R}^{i}(f,M))$ is artinian by \cite[Lemma 2.1]{M6}.
This means that $\mathrm{kerExt}_{R}^{i}(f,M)$ is minimax, and hence $\mathrm{kerExt}_{R}^{i}(f,M)$ is $I$-coartinian  for $0\leq i\leq n+1$ by \cite[Lemma 3.5]{N1}. Furthermore, from the exact sequence
$$0\rightarrow \mathrm{cokerExt}_{R}^{i}(f,M)\rightarrow \mathrm{Ext}_{R}^{i}(U,M)\rightarrow \mathrm{kerExt}_{R}^{i+1}(f,M)\rightarrow 0,$$
one has $\mathrm{cokerExt}_{R}^{i}(f,M)$ is $I$-coartinian for $0\leq i\leq n$. Consequently, the exact sequence
$$\mathrm{cokerExt}_{R}^{i}(g,M)\rightarrow \mathrm{cokerExt}_{R}^{i}(gf,M)\rightarrow \mathrm{cokerExt}_{R}^{i}(f,M)\rightarrow 0$$
induces that $\mathrm{cokerExt}_{R}^{i}(gf,M)=\mathrm{Ext}_{R}^{i}(N,M)/x\mathrm{Ext}_{R}^{i}(N,M)$ is $I$-coartinian for $0\leq i\leq n$. Hence \cite[Proposition 2.6]{S2} implies that the $R$-module $\mathrm{Ext}_{R}^{i}(N,M)$ is $I$-coartinian for $0\leq i\leq n$. This completes the inductive step.
\end{proof}

By Theorem \ref{thm:A}, \cite[Corollary 2.7]{BN} and \cite[Theorem 3.6]{S2}, we have the following immediate corollary.

\begin{cor}\label{cor:2.8}
Let $N$ be a finitely generated $R$-module with $\mathrm{dim}_{R}N/IN\leq 1$ and $M$ a semi-discrete linearly compact $R$-module with $\mathrm{mag}_{R}M\leq 1$. Then $\mathrm{Ext}_{R}^{i}(\mathrm{H}_{I}^{j}(N),\mathrm{H}^{I}_{k}(M))\in \mathcal{C}^{1}(R,I)_{coa}$ for all $i,j,k\geq 0$.
\end{cor}

We next study the coartinianess of the $R$-module $\mathrm{Tor}_{i}^{R}(N,M)$.

\begin{lem}\label{lem:2.0}
Let $M$ be an $I$-coartinian $R$-module. Then for any finite length $R$-module $N$, the $R$-module
$\mathrm{Tor}_{i}^{R}(N,M)$ has finite length for $i\geq 0$.
\end{lem}
\begin{proof}
The assertion is clear since $\mathrm{Cosupp}_{R}\mathrm{Tor}_{i}^{R}(N,M)\subseteq \mathrm{Supp}_{R}N\cap \mathrm{Cosupp}_{R}M\subseteq \{\mathfrak{m}\}$ and $\mathrm{Tor}_{i}^{R}(N,M)$ is artinian by \cite[Corollary 3.6]{N} and \cite[Theorem 2.2.10]{F}.
\end{proof}

\begin{lem}\label{lem:2.8}
Let $M$ be a semi-discrete linearly compact $I$-coartinian $R$-module and $N$ a finitely generated $R$-module with $\mathrm{dim}_{R}N=1$. Then $\mathrm{Tor}_{i}^{R}(N,M)\in \mathcal{S'}$ for $i\geq 0$.
\end{lem}
\begin{proof}
Let $P_\bullet\rightarrow N$ be a free resolution by finitely generated free $R$-modules. As submodules and homomorphic images of a semi-discrete linearly compact module are also semi-discrete linearly compact, one has
$\mathrm{Tor}^R_i(N,M)=\mathrm{H}_i(P_\bullet\otimes_RM)$ is semi-discrete linearly compact, and so $\mathrm{Tor}_{i}^{R}(N,M)$ is minimax for $i\geq 0$ by \cite{Z}. We only need to show that $\mathrm{Tor}_{i}^{R}(N,M)$ is $I$-coartinian for $i\geq 0$. Consider the exact sequence
$$0\rightarrow \Gamma_{I}(N)\rightarrow N\rightarrow N/\Gamma_{I}(N)\rightarrow 0,$$
which induces the exact sequence of linearly compact $R$-modules
$$\cdots\rightarrow \mathrm{Tor}_{i}^{R}(\Gamma_{I}(N),M)\rightarrow \mathrm{Tor}_{i}^{R}(N,M)\rightarrow \mathrm{Tor}_{i}^{R}(N/\Gamma_{I}(N),M)\rightarrow \cdots.$$
Since $\Gamma_{I}(N)$ is finitely generated,  $\mathrm{Tor}_{i}^{R}(\Gamma_{I}(N),M)$ is artinian for $i\geq 0$ by \cite[Theorem 2.2.10]{F}. Therefore, $\mathrm{Tor}_{i}^{R}(\Gamma_{I}(N),M)$ is minimax $I$-coartinian for $i\geq 0$. If  $\mathrm{Tor}_{i}^{R}(N/\Gamma_{I}(N),M)$ is also minimax $I$-coartinian for $i\geq 0$, then the assertion is true for $N$. Consequently, without loss of generality, we may suppose that $\Gamma_{I}(N)=0$. Then $I\nsubseteq \bigcup_{\mathfrak{p}\in \mathrm{Ass}_{R}N}\mathfrak{p}$, and there exists an element $x\in I$ is a non-zero divisor of $N$. Now, the exact sequence
$$\xymatrix{0\ar[r] & N\ar[r]^{x} & N\ar[r] & N/xN\ar[r] & 0}$$
induces the exact sequence of linearly compact $R$-modules
$$0\rightarrow \mathrm{Tor}_{i}^{R}(N,M)/x\mathrm{Tor}_{i}^{R}(N,M)\rightarrow \mathrm{Tor}_{i}^{R}(N/xN,M).$$
Since $N/xN$ is finitely generated and $\mathrm{dim}_{R}N/xN=0$, it has finite length. Hence Lemma \ref{lem:2.0} implies that $\mathrm{Tor}_{i}^{R}(N/xN,M)$ has finite length for $i\geq 0$, and hence the $R$-module $\mathrm{Tor}_{i}^{R}(N,M)/I\mathrm{Tor}_{i}^{R}(N,M)$ is of finite length for $i\geq 0$. Consequently, $\mathrm{Tor}_{i}^{R}(N,M)$ is $I$-coartinian for $i\geq 0$  by \cite[Lemma 3.5]{N1}.
\end{proof}

\begin{proof}[Proof of Theorem \ref{thm:B}]
As in the proof of Lemma \ref{lem:2.8}, we may assume that $\Gamma_{I}(N)=0$ and $\mathrm{dim}_{R}N=2$. Then $I\nsubseteq \bigcup_{\mathfrak{p}\in \mathrm{Ass}_{R}N}\mathfrak{p}$, and there exists an element $x\in I$ such that $x\notin \bigcup_{\mathfrak{p}\in \mathrm{Ass}_{R}N}\mathfrak{p}$. Then the exact sequence
$$\xymatrix{0\ar[r] & N\ar[r]^{x} & N\ar[r] & N/xN\ar[r] & 0}$$
induces the exact sequence of linearly compact $R$-modules
$$\xymatrix{\mathrm{Tor}_{i+1}^{R}(N/xN,M)\ar[r] & \mathrm{Tor}_{i}^{R}(N,M)\ar[r]^{x} & \mathrm{Tor}_{i}^{R}(N,M)\ar[r] & \mathrm{Tor}_{i}^{R}(N/xN,M).}$$
Since $N/xN$ is finitely generated and $\mathrm{dim}_{R}N/xN=1$, $\mathrm{Tor}_{i}^{R}(N/xN,M)\in \mathcal{S'}$ for $i\geq 0$ by Lemma \ref{lem:2.8}, and so $(0:_{\mathrm{Tor}_{i}^{R}(N,M)}x)$ and $\mathrm{Tor}_{i}^{R}(N,M)/x\mathrm{Tor}_{i}^{R}(N,M)$ are also in $\mathcal{S'}$ for $i\geq 0$. Therefore, $\mathrm{Tor}_{i}^{R}(N,M)$ is $I$-coartinian for $i\geq 0$ by \cite[Proposition 2.6]{S2}.
\end{proof}

\bigskip \centerline {\bf Acknowledgments}
This research was partially supported by National Natural Science Foundation of China (11761060, 11901463).

\end{document}